\documentclass[a4paper,11pt,twoside,DIV=calc]{scrartcl}	

\input xy
\xyoption{all}
\usepackage{lscape}
\usepackage[latin1]{inputenc}
\usepackage[cmtip,arrow]{xy}
\usepackage{amsmath}
\usepackage{amssymb,amscd,amsthm,mathrsfs,yfonts,manfnt,pb-diagram,pb-xy}
\usepackage[nice]{nicefrac}		
\usepackage[T1]{fontenc}
\usepackage{lmodern} 					
\usepackage{tocenter}					
\ToCenter[h]{\textwidth}{\textheight}	
\usepackage{ulem}							
\usepackage{color}

\renewcommand{\emph}{\textit}		

\usepackage{ifthen}
\newboolean{comm}
\setboolean{comm}{false}
\newboolean{sol}
\setboolean{sol}{false}
\newboolean{notes}
\setboolean{notes}{false}			
\CompileMatrices  						

\newcommand{\com}{\ifthenelse{\boolean{comm}}}
\newcommand{\sol}{\ifthenelse{\boolean{sol}}}
\newcommand{\note}{\ifthenelse{\boolean{notes}}}

\newtheorem{Def}{Definition}
\newtheorem{Prop}[Def]{Proposition}

\newtheorem{Th}[Def]{Theorem}

\newtheorem{Lem}[Def]{Lemma}
\newtheorem{Cor}[Def]{Corollary}

\theoremstyle{definition}

\newcommand{\mK}{\ensuremath{\mathbb{K}}}

\newcommand{\mB}{\ensuremath{\mathbb{B}}}

\newcommand{\mc}{\mathcal}								

\newcommand{\ee}{{\mbox{$\varepsilon$}}}

\renewcommand{\phi}{\varphi}

\newcommand{\mb}{\begin{pmatrix}}					
\newcommand{\me}{\end{pmatrix}}						
\newcommand{\hsp}{\mc{H}}									
\newcommand{\bo}{\mathbb{B}(\mc{H})}			
\newcommand{\co}{\mathbb{K}(\mc{H})}			
\newcommand{\Mul}{\mathcal{M}}						

\DeclareMathOperator{\Ad}{Ad}							

\DeclareMathOperator{\id}{id}							

\newcommand{\rrarrow}{\rightrightarrows}

\usepackage{tocenter}				
\ToCenter[h]{\textwidth}{\textheight}	

\setkomafont{section}{\bfseries\Large}		

\usepackage{scrpage2}      	
\ofoot{}					
\ifoot{}
\rehead{\pagemark}			
\cehead{MARTIN GRENSING}	
\rohead{\pagemark}			
\cohead{A note on Kasparov products}

\begin{document}

\title{\Large{\MakeUppercase{A note on Kasparov products}}}
\author{\small{MARTIN GRENSING}}
\date{\small{\today}}
\maketitle
\thispagestyle{empty}
\begin{abstract} \small{Combining Kasparov's theorem of Voiculesu and Cuntz's description of $KK$-theory in terms of quasihomomorphisms, we give a simple construction of the Kasparov product. This will  be used in a more general context of locally convex algebras in order to treat products of certain universal cycles.}
\end{abstract}
 \section{Introduction}
The goal of this note is to establish existence of the Kasparov product based on Kasparov's theorem of Voiculescu (\cite{MR587371}), and to examine how this construction is related to the one used by Kasparov.

In the first section, we interpret the connection condition and the existence of Kasparov product (\cite{KaspOp}) as the existence of a certain extension of a quasihomomorphism (\cite{MR899916}). Such extensions always exist, as can be seen by applying split exactness of $KK$ to a certain algebra $D_\alpha$ that is a semidirect product of the domain and target of a quasihomomorphism. The resulting description of the Kasparov product already yields a useful way to construct the Kasparov product; it is particularly well adapted to generalisations of the bimodule-formalism to locally convex algebras, where it may be used to calculate products of certain "smooth" submodules, and is used in \cite{UC} in a crucial manner. 

In the second section, it is shown that, without making use of split exactness of $KK$, one can, in case that Kasparov's theorem of Voiculscu is available, construct the product by using this interpretation. First we show how to reduce quasihomomorphisms to a single morphism and a unitary; and if an absorbing morphism is chosen, all classes of quasihomomorphisms are obtained from it by conjugation by a unitary. Applying this to a pair of composable quasihomomorphisms, we see that it suffices to extend quasihomomorphisms to just one "universal" algebra; if further the domain or target of the first quasihomomorphism is nuclear, there is a canonical way to extend quasihomomorphisms.

I would like to thank G. Skandalis for  for many remarks and fruitful discussions, and for sharing with me his insights and ideas concerning mathematics and $KK$-theory in particular.
\section{The Kasparov product revisited}
Quasihomomorphisms were introduced by Cuntz in \cite{MR733641} and further developed in \cite{MR899916}.
\begin{Def} Let $B$ be stable, $\hat B$  a $C^*$-algebra containing $B$ as an ideal; then a  quasihomomorphism from $A$ to $B$ is a pair of homomorphisms  from $A$ to $\hat B$ such that $\alpha(a)-\bar\alpha(a)\in B$ for all $a\in A$. 

For nonstable $B$, a quasihomomorphism from $A$ to $B$ is by definition a  quasihomomorphism from $A$ into the stabilisation $\mathbb{K}\otimes B$ of $B$. 
\end{Def}
 Let $(E,\phi,F)$ be a Kasparov $(A,B)$-module with $A$ and $B$ trivially graded. If $F$ is selfadjoint and invertible, then with respect to the grading:
$$\phi=\mb \phi^{(0)}&\\&\phi^{(1)}\me\text{ and } F=\mb &T^{-1}\\T&\me$$
where the $\phi^{(i)}$  are homomorphisms $A\to \mathbb{B}_{B}(E^{(i)})$ and $T$ is by hypothesis a unitary in $\mathbb{B}_{B}(E^{(0)},E^{(1)})$. Thence we obtain a quasihomomorphism $(\alpha,\bar\alpha):=(\phi^{(0)},T^{-1}\phi^{(1)} T)$ from $A$ to $\mathbb{K}_B(E^{(0)})$ simply by identifying $E^{(0)}$ and $E^{(1)}$ via $T$, and where we view the latter as a subalgebra of $\mathbb{K}\otimes B$ via the stabilization-theorem.

We may always reduce to this case by using the standard simplifications in $KK$-theory, and therefore we can define an associated quasihomomorphism $Qh(x)$ to every Kasparov module.

The original construction of the Kasparov product from \cite{KaspOp} was quite technical. We will use the version  based on the notion of connection introduced by Connes and Skandalis. We fix the following setting: Let $E_1$ be a graded Hilbert $B$-module, $E_2$ a graded Hilbert $C$-module, $\phi:B\to\mathbb{B}_C(E_2)$ a $*$-homomorphism and $F$ an odd selfadjoint operator on $E_2$. We set $E_{12}:=E_1\otimes_B E_2$, and define for every $x\in E_1$ an operator $T_x:E_2\to E_1\otimes_B E_2,\; y\mapsto x\otimes y$. Note that the adjoint of $T_x$ is given by the mapping $E_{12}\to E_2,\; y\otimes z\mapsto \phi(\langle x|y\rangle) z$, and $T_{x'} T_x^*=\theta_{x',x}\otimes \id_{E_2}$.
\begin{Def}
An $E_1$-connection for an odd operator $F$ is an odd salfadjoint operator $G$ such that for all homogeneous $x\in E_1$
$$T_x F-(-1)^{\partial x} G T_x\in \mathbb{K}_C(E_2,E_{12})\text{ and } F T_x^*-(-1)^{\partial x} T_x^* G\in\mathbb{K}_C(E_{12},E_2).$$
\end{Def}
 As a consequence of the stabilisation theorem, such connections exist in case we deal with Kasparov modules; more precisely:
\begin{Prop} If $E_1$ is countably generated and $[F,b]$ is compact for all $b\in B$, then there exists an odd $E_1$ connection for $F$. 

If $(E_1,\phi_1,F_1)$ is a Kasparov $(A,B)$-module, $(E_2,\phi_2,F_2)$ a Kasparov $(B,C)$-module, $G$ an $F_2$ connection for $E_1$, then $(E_{12},\id_{\mathbb{K}_B(E_1)}\otimes 1,G)$ is a Kasparov $(\mathbb{K}_B(E_1),C)$-module.
\end{Prop}
The existence statement stems from \cite{MR775126}; the second fact was stated in \cite{MR743845}, Proposition 9. 
\ifthenelse{\boolean{notes}}{\marginpar{We may also construct the connection as follows: (use the Morita-equivalence stuff)}}{}

The composition product is given in terms of the representatives of the cycles involved: If $(\phi_1,E_1,F_1)$ is a Kasparov $(A,B)$-module and $(\phi_2,E_2,F_2)$ a Kasparov $(B,C)$-module, then a Kasparov $(A,C)$-module $(E_1\otimes_B E_2,\phi_1\otimes 1,F_{12})$ is called a product of $(E_1,\phi_1,F_1)$ and $(E_2,\phi_2,F_2)$ if 
\begin{enumerate}
\item $F_{12}$ is an $E_1$ connection for $F_2$ (connection condition)
\item For all $a\in A$, $\phi_1(a)\otimes 1[F_1\otimes 1,F_{12}]\phi_1(a)^*\otimes 1$ is positive in the quotient  $\mathbb{B}_C(E_{12})/\mathbb{K}_C(E_{12})$  (positivity condition).
\end{enumerate}
The set of operators $F_{12}$ satisfying the above conditions will be denoted $F_1\sharp F_2$.

Using Kasparov's technical theorem, one can show that a product as above always exists if $A$ is separable, is unique up to operator homotopy, and passes to homotopy classes (cf. \cite{MR743845}).

Recall also that a Hilbert $B$-module $E$ is called full if the linear span of $\langle E|E\rangle$ is dense in $B$.  
\begin{Def} Let $A$ and	 $B$ be graded $C^*$-algebras. A graded Morita(-Rieffel) equivalence between $A$ and $B$ is given by a graded full Hilbert $B$-module $E$, called the equivalence bimodule, and a graded isomorphism $\phi:A\to \mathbb{K}_B(E)$. 
\end{Def}
We identify $A$ with $\mathbb{K}(E)$ and drop the isomorphism $\phi$. If $E$ is a graded Morita equivalence bimodule from $A=\mathbb{K}_B(E)$ to $B$, then we define the $(B,\mathbb{K}(E))$-module $E^*:=\mathbb{K}(E,B)$. The $\mK(E)$-valued scalar product is simply $\langle T|S\rangle:=R^*S$, and this makes $E^*$ into a graded Hilbert $\mK(E)$-module.

Let $A$ and $B$ be separable. Then the class $[(E,\id_A,0)]$ of the equivalence bimodule  yields a $KK$ equivalence from $A$ to $B$ with inverse $[(E^*,\id_B,0)]$.

Conversely, any given full Hilbert $B$-module $E$ may be viewed as a graded Morita equivalence from $\mathbb{K}_B(E)$ to $B$.

If $y=(E_2,\phi_2,F_2)\in \mathbb{E}(B,C)$, $E_1$ is a Hilbert $B$-module, and $w$ denotes the Kasparov module defined by the Morita equivalence determined by $E_1$, then the operator $G$ in a product $w\cap x$ is exactly an $E_1$-connection for $F_2$, as the positivity condition is trivially satisfied.

If $x=(E_1,\phi_1,F_1)$ and $v=(E_1^*,\id_B,0)$ is the inverse of $w$, then the product $x\cap v$ is represented by $(\mathbb{K}_B(E_1),\phi_1,F_1)$, where the bounded operators on $E$ are considered to act on the Hilbert $\mathbb{K}_B(E_1)$-module by multiplication. This is easily seen by using the explicit form of the isomorphism $U:E_1\hat \otimes_B E_1^*\to \mathbb{K}_B(E_1)$ given above, as  $UTU^{-1}(|\xi\rangle\langle\eta|)=|T\xi\rangle \langle\eta|$ for all $T\in\mathbb{B}_B(E)$. Hence compact operators on $E$ act again by compact operators on $\mathbb{K}_B(E)$, and therefore $(\mathbb{K}_B(E_1),\phi,F_1)$ does indeed define a cycle, the connection condition is obvious, and positivity follows from  $a[F_1,F_1]a^*=a(2F_1^2)a^*=aa^*$ modulo compacts.

We fix two Kasparov bimodules $(E_1,\phi_1,F_1)\in \mathbb{E}(A,B)$ and $(E_2,\phi_2,F_2)\in\mathbb{E}(B,C)$, and denote their classes in $KK$ by $x$ and $y$. The module $E_1$ is seen as a Morita equivalence from $\mathbb{K}_B(E_1)$ to $B$, whose class in $KK$ we denote by $w$, and its inverse by $v$. Let $(\alpha,\bar \alpha):A\rrarrow D\trianglerighteq \mathbb{K}_B(E_1)$ be the quasihomomorphism associated to $x':=x\cap v$, and recall that $y':=w\cap y$ may be viewed as the class of the Kasparov module defined via an $E_1$ connection for $F_2$. If we define $D_\alpha$ as the sub-$C^*$-algebra of $A\oplus D$ generated by $(a,\alpha(a))$ and $0\oplus B$, $a\in A$, we obtain the double split short exact sequence
\[\xymatrix{0\ar[r]&\mathbb{K}_B(E_1)\ar[r]^-\iota\ar[r]&D_\alpha\ar[r]& A\ar[r]\ar@/^.4cm/[l]^{\id_A\oplus\alpha}\ar@/_.4cm/[l]_{\id_A\oplus\bar\alpha}&0}\]
which in turn, by split exactness of $KK$, yields a long exact sequence 
$$\xymatrix{0\ar[r]&KK(A,C)\ar[r]&KK(D_\alpha,C)\ar[r]^-{\iota^*}&KK(\mathbb{K}_B(E_1),C)\ar[r]&0}.$$
We may thus assume that $y'=\iota^* z$ for some $z\in KK(D_\alpha,C)$. We claim that $\alpha^*(z)-\bar\alpha^*(z)=y\cap x$. This follows as $KK((\alpha,\bar\alpha),C)$ is multiplication by $x'$ on the left, and therefore
$$x\cap y=x'\cap y'=KK((\alpha,\bar\alpha),C)(y')=(\alpha^*-\bar\alpha^*)(\iota^*)^{-1}\iota^*(z)=(\alpha^*-\bar\alpha^*)(z).$$
Calculating a representative for the last expression, we have thus proved:
\begin{Th}\label{productviaext} Let $x\in KK(A,B)$, $y=[(E_2,\phi_2,F_2)]\in KK(B,C)$. Then the Kasparov product of $x$ and $y$ may be defined by
\begin{enumerate}
\item representing $x$ as a quasihomomorphism $(\alpha,\bar\alpha): A\rrarrow\mathbb{B}_B(E_1)\trianglerighteq \mathbb{K}_B(E_1)$.
\item choosing an $E_1$ connection $G$ for $F_2$
\item\label{ext} lifting the Kasparov $(\mathbb{K}_B(E_1),C)$-module $(E_1\otimes_B E_2,\id_{\mathbb{K}_B(E_1)}\otimes 1,G)$ along the canonical inclusion of $\mathbb{K}_B(E_1)\to D_\alpha$ to a Kasparov $(D_\alpha,C)$-module $(\tilde\phi,\tilde E,\tilde G)$,
\item and setting 
$$x\cap y:=\left[ \left(\mb \tilde\phi\circ\alpha&\\&\tilde\phi\circ\bar\alpha\circ \ee\me,\tilde E\oplus \tilde E^{op},\mb \tilde G& \\ &-\tilde G\me\right)\right] \in KK(A,C)$$
where $E^{op}$ denotes the Hilbert $B$-module $E$ with inversed grading, and $\ee$ the grading operator on $E$.
\end{enumerate}
\end{Th}

Here  \ref{ext} means exactly that the quasihomomorphism 
$$Qh(E_1\otimes_B E_2,\id_{\mathbb{K}_B(E_1)}\otimes 1,G)$$ 
extends to a quasihomomorphism on the larger algebra $D_\alpha$; note that the class of the cycle $x\cap y$ as defined above is independent of the choice of the extension.

\section{Reduction of quasihomomorphisms and a construction of the Kasparov product}

For a given linear map $\phi:A\to \mB_B(E)$, where $E$ is a Hilbert $B$-module, we define $E^\infty:=\bigoplus_{n=1}^\infty E$\index{$E^\infty$}, and $\phi^\infty:A\to \mB_B(E^\infty)$\index{$\phi^\infty$} as the diagonal action of $\phi$.

\begin{Prop}\label{quasipresent} The class of every quasihomomorphism is represented by a quasihomomorphism of the form $(\alpha,\Ad_U\circ\alpha)$, where $U$ is a unitary.
\end{Prop}
\begin{proof}
Let $(\alpha,\bar\alpha):A\rrarrow\hat B\trianglerighteq B$ be a quasihomomorphism. We may assume that $\hat B=\mB_B(E)$ and $B=\mK_B(E)$ for some Hilbert $B$-module $E$. We may replace $(\alpha,\bar\alpha)$ by
$$(\alpha\oplus\alpha^\infty\oplus\bar\alpha^\infty,\alpha\oplus\alpha^\infty\oplus\bar\alpha^\infty):A\to \mB_B(E\oplus E^\infty\oplus E^\infty)\trianglerighteq \mK_B(E\oplus E^\infty\oplus E^\infty)$$
because $(\alpha^\infty\oplus\bar\alpha^\infty,\alpha^\infty\oplus\bar\alpha^\infty)$ is degenerate.

Now let $U$ be the unitary on $E\oplus E^\infty\oplus E^\infty$ that maps
$$(\xi_0,(\xi_1,\xi_2,\ldots),(\eta_1,\eta_2,\ldots))\to(\xi_1,(\xi_2,\xi_3,\ldots),(\xi_0,\eta_1,\eta_2,\ldots)).$$
Then 
$$(\alpha(a)\oplus\alpha^\infty(a)\oplus\bar\alpha^\infty(a))U=U(\bar\alpha(a)\oplus\alpha^\infty(a)\oplus\bar\alpha^\infty(a)).$$
\end{proof}

\begin{Def} Let $A$ and $B$ be $C^*$-algebras, $\beta:A\to\mB(\hsp_B)$ a $*$-homomorphism such that for every $*$-homomorphism $\alpha:A\to\mB(\hsp_B)$ there exists a unitary $U$ with $\alpha\oplus\beta=U^*\beta U$ modulo compact operators. Then $\beta$ will be called absorbing.
\end{Def} 
The following theorem was proved in \cite{MR587371}:
\begin{Th}[Kasparov-Voiculescu] Let $A$ and $B$ be separable $C^*$-algebras and $\beta_0:A\to\bo$ a faithful representation of $A$ such that $(\tilde\beta_0)^{-1}(\co)=\{0\}$. We denote by $\beta$ the inclusion of $A$ into $\mB_B(\hsp_B)$ obtained from $\beta_0$  by viewing $\mB(\hsp)$ as a subalgebra of $\mB_B(\hsp_B)$.   If either $A$ or $B$ is nuclear, then $\beta$ is absorbing.
\end{Th}

In general, there is a result of Thomsen from \cite{MR1814167}, Theorem 2.7, which shows that for $A$ and $B$ separable, there is an absorbing homomorphism from $A$ into the stable multiplier algebra $\Mul(B\otimes\mK)$ of $B$. 

\begin{Lem} Let $(\alpha,\alpha^U):A\rrarrow \hat B\trianglerighteq B$ be a quasihomomorphism, and $\beta:A\to\hat B$ a homomorphism such that $\alpha(a)-\beta(a)\in B$ for all $a$. Then $(\beta,\beta^U)$ is a quasihomomorphism equivalent to $(\alpha,\alpha^U)$.
\end{Lem}
\begin{proof} Using the usual rotation matrices, we obtain a path of unitaries
$$U_t:=\mb \cos(t)&-\sin (t)\\\sin(t)&\cos(t)\me\mb U&\\&1\me\mb \cos(t)&\sin (t)\\ -\sin(t)&\cos(t)\me.$$
Reparametrizing, we get a homotopy 
$$(\alpha\oplus\beta,\Ad_{U_t} \circ\;\alpha\oplus\beta)$$
of the quasihomomorphisms $(\alpha,\Ad_U \circ\alpha)\oplus (\beta,\beta)$ and $(\alpha,\alpha)\oplus(\beta,\Ad_U\circ\beta)$
\end{proof}

\begin{Prop}\label{presentquasi} Let $\beta:A\to\mB(\hsp_B)$ be absorbing. Then  every element of $KK(A,B)$ is represented by a quasihomomorphism   of the form $$(\beta,\Ad_U\circ\beta):A\to\mB_B(\hsp_B)\trianglerighteq \mK\otimes B,$$ where $U\in\mB_B(\hsp_B)$ is a unitary.
\end{Prop}
\begin{proof}
By Proposition \ref{quasipresent}, we may assume that we are given a quasihomomorphism $(\alpha,\alpha^U):A\rrarrow \mB_B(\hsp_B)\trianglerighteq \mK_B(\hsp_B)$, where $U$ is a unitary in $\mB_B(\hsp_B)$.   Let $V$ be a unitary such that $\alpha\oplus \beta=V^*\beta V$. Then we get
\begin{align*}(\alpha,\alpha^U)\sim(\alpha\oplus\beta,\alpha^U\oplus \beta)\sim(\beta^V,\beta^{V(U\oplus 1)})\sim(\beta,\beta^{V(U\oplus 1)V^*})
\end{align*}
by the above Lemma.
\end{proof}
\begin{Cor} Let $A$, $B$, $C$ be separable $C^*$-algebras, $\beta$ as in the above Proposition absorbing, $(\gamma,\bar\gamma)$ a quasihomomorphism from $B$ to $C$. Then it suffices to find an extension of $(\gamma,\bar\gamma)$ to the one algebra $D_\beta$,  in order to calculate explicitly \underline{all} products of $(\gamma,\bar\gamma)$ with elements from $KK(A,B)$ (as in  \ref{productviaext}).
\end{Cor}

One can use these ideas to construct the Kasparov product in good cases:

Let $(\alpha,\bar\alpha):A\rrarrow \mathbb{B}(B\otimes \hsp)\trianglerighteq B\otimes \co$ be a quasihomomorphism, where $\bar\alpha=1\otimes \pi$ is induced by a representation $\pi$ of $A$ on $\hsp$ with $\pi(A)\cap \co=\{0\}$. Let $(\beta,\bar\beta):B\rrarrow \hat C\trianglerighteq C$ be another quasihomomorphism. We may extend $(\beta,\bar\beta)$ to a quasihomomorphism
$$(\beta',\bar\beta'):1\otimes\bo+B\otimes\co\to\Mul(\hat C\otimes \co)\trianglerighteq  C\otimes\co$$ by first stabilizing and then setting $\beta'(1\otimes T+x):=1\otimes T+\beta\otimes\id_{\mK}(x)$. Because $D_{\bar\alpha}\subseteq 1\otimes\bo+B\otimes\co$, we have constructed a product. Note further that because $(\beta',\bar\beta')$   represents zero on the image of $\bar\alpha$, the product has a very simple form:
$$[\alpha,\bar\alpha]\,[\beta',\bar\beta']=[\beta'\circ\alpha,\bar\beta'\circ\alpha].$$

In particular, if we have any two quasihomomorphisms $(\alpha,\bar\alpha)$ from $A$ to $B$ and $(\beta,\bar\beta)$ from $B$ to $C$ and either  $A$ or $B$ is nuclear, then by Proposition \ref{presentquasi}  we may assume that $\bar\alpha$ is obtained from a faithful representation $A$ whose image is disjoint from the compacts, and then apply the construction as above. More generally, one may construct on this way the Kasparov product for the functor $KK_{nuc}$ from \cite{MR953916}.

This construction of the product coincides with the one by Kasparov by the preceeding section.

\bibliographystyle{alpha}
\bibliography{../../Fullbib}
\end{document}